\documentclass{elsarticle}
\usepackage{geometry}
\usepackage{bm}
\usepackage{amsmath}
\usepackage{amsthm}
\usepackage{amssymb}
\usepackage{tikz}
\usepackage{mathtools}

\usepackage{float}
\usepackage{tabularx}

\usepackage{caption,subcaption}
\usepackage{color}

\usepackage{tikz}

\newtheorem{lemma}{\textbf{Lemma } }
\newtheorem{definition}{\textbf{Definition} }

\newtheorem{remark}{\textbf{Remark} }

\newtheorem{conjecture}{\textbf{Conjecture} }
\begin{document}
	
	\title{
		Robust  Eigenvectors of  Regular  Simplex Tensors:    Conjecture Proof
		\tnoteref{mytitlenote}
	}
	\tnotetext[mytitlenote]{
		This work was partially supported by National Natural Science Fund of China (62271090), Chongqing Natural Science Fund (cstc2021jcyj-jqX0023), National Key R\&D Program of China (2021YFB3100800), CCF Hikvision Open Fund (CCF-HIKVISION OF 20210002), CAAI-Huawei MindSpore Open Fund, and Beijing Academy of Artificial Intelligence (BAAI).
	}
	
	\author[add1]{Lei  Wang\corref{coraut}}
	\cortext[coraut]{Corresponding author}
	\ead{wanglei179@mails.ucas.ac.cn}
	\author[add2,add3,add4]{Xiurui Geng}
	\author[add1]{Lei  Zhang}

	\address[add1]{	the School of Microelectronics and Communication
		Engineering, Chongqing University, Chongqing 400044, China}
	\address[add2]{Aerospace Information Research Institute, Chinese Academy of Sciences, Beijing 100094, China}
	\address[add3]{University of the Chinese Academy of Sciences, Beijing 100049, China} 
	\address[add4]{Key Laboratory of  Technology in Geo-Spatial Information Processing and Application System, Chinese Academy of Science, Beijing 100190, China}


\begin{abstract}
The  concept of   tensor  eigenpairs   has  received  more researches   in  past  decades. 
Recent works have   paid  attentions to   a  special class of  symmetric tensors  termed  regular simplex tensors,  which is  constructed by  
equiangular tight frame  of $n+1$ vectors in 
$n$-dimensional space,  and    the  robustness of eigenpairs was investigated. 
In  the end of the literature,  a  conjecture  was    claimed  that the robust eigenvectors of a regular simplex tensor  are precisely
 the
vectors in the frame.
One  later
work 
theoretically 
proved that 
the case of $n=2$ was true. 
In this paper, 
we  proceed further   
and  complete  the proof for  the above conjecture.
Some    promising directions 
are  discussed in the end   for    future works.

\end{abstract}


\maketitle


\section{Introduction}


Tensor   analysis and   applications  have been researched 
in recent years, and 
among them, 
the concept of tensor eigenpairs 
 has been widely exploited  in  theory \cite{SHOPM,NCM,OTD,Cuicf,hm}  and also  
made 
numerous  applications   in  many disciplines, such as latent  variable mode \cite{Hsu}, hyperspectral  image processing \cite{NPSA,MSDP}, 
signal processing \cite{tensor_bss1}  and so on.


However,  it  has  also  been   shown  that  most of tensor  problems  are NP hard  \cite{NP-hard},  including computing all  eigenpairs of tensors.
To our best knowledge,  
so far,  there are  only  two  algorithms that  can  obtain  all  eigenpairs of tensor \cite{Cuicf,hm}.
Therefore, most of  previous works aim
 to obtain the maximized or minimized one eigenpairs, and different optimization algorithms were developed, such as \cite{SHOPM,ASHOPM,NCM}.
 Among them, 
 one of the 
 classical algorithms
 is called  tensor power method (TPM),  which is based on a  fixed-point scheme. 
 The  fixed-points (may be more than one) of the  TPM 
 are the eigenvectors of tensor.
 An important  issue concerning the TPM  method 
 is to 
 investigate 
 that the   obtained  eigenvector  is  robust or not (See Subsection \ref{robustsection} for details).
Concerning  this problem, 
previous works have 
mainly  paid attentions to some   special  classes  of  symmetric  tensors,  
and  one   widely researched   type  
is termed 
  orthogonally  decomposable  (odeco) tensors \cite{odst,OTD,Hsu,sr1,cunmu,GloballyConvergent},  
  which 
  is  
  an natural generalization of orthogonal matrix decomposition.
A   good  property for odeco tensors
is  that  the  classical TPM  can  exactly 
extract 
the vectors 
that  is 
the generator of the odeco  tensor.
In  other words,
the robust eigenpairs of  such a type of tensor
can be found by 
TPM.


However, 
unfortunately, 
most of symmetric tensors cannot be 
orthogonally  decomposable. 
In  this sense, 
recently, 
some researchers further  extended 
the case of  odeco tensors into 
a  more generalized one, where the  symmetric  tensor  is  generated  by 
the set of some  equiangular  set (ES) or equiangular tight frame (ETF), which contains 
 of $r$ vectors in 
$n$-dimensional space \cite{RobustEigen}. 
For  example, 
the case of  $r=n$  serves as a  special one of tight frame,  which  
forms a  standard orthonormal  basis and  corresponds to 
the  odeco tensors.
When $r=n+1$,  the 
frame is termed the regular simplex one, 
and the  generated tensor is thus called regular simplex tensor. 
In  \cite{RobustEigen}, 
they  discussed that under what  condition 
the eigenvectors will be a robust one, 
and 
studied 
 these  special  types of  tensors    generated by ES or ETF with  some theoretical  proof. 
 
 Furthermore, 
in  the end of the literature \cite{RobustEigen},  a  conjecture  was    claimed  that the robust eigenvectors of a regular simplex tensor  are precisely
the
vectors in the frame (See  Conjecture   \ref{conjecturesimplex}   in Subsection \ref{conjecture} for formal statement).
Later, 
several researchers 
further 
focused  on  this conjecture.  
By reformulating  
the regular  simplex tensor eigenpairs equation 
as an algebraic   system of 
equations,
they 
  investigated 
the real eigen-structure 
 of  all eigenpairs  and tended to  identify which ones are robust\cite{teneigenstructure}.
However, 
they only provided the robustness proof for the case of  $n=2$,
and the experimental results  justification for  the case of  $n=3$.
For more bigger $n$ cases,  it 
will be a tough task for proof.
There are  mainly two difficulties concerning such a proof way:
1): as $n$   becomes larger, 
the number of all eigenpairs  exponentially increases, and 
determining  the spectral radius for all of them is computationally 
heavy; 
2):  the eigen-structure of all 
eigenpairs is also complicated, 
and  thus  it could be  tough   for most  of  eigenpairs to 
calculate  the Jacobian matrix and let alone 
 determine its spectral radius.
In this paper, in the basis of the only two works \cite{RobustEigen,teneigenstructure}
and to avoid these  difficulties,
we proceed further and provide a complete proof   concerning this  conjecture 
and 
the main contributions of this paper are 
concluded as  follows:

The  locally maximized eigenpairs of a tensor  is another important concept 
in the optimization theory.
The connection between robust 
and locally maximized eigenpairs  
was  firstly  investigated, and 
we theoretically show that 
robust eigenpairs is  subset of all  these  locally maximized ones. 
Such a  conclusion enables us to simplify the proof way adopted in \cite{teneigenstructure}.
In detail, one does not need to check 
the robustness of  all eigenpairs, and   
it will be enough to only  focus on these  locally maximized eigenpairs.
This is our main finding to deal with the robustness problem, 
and the proof for the unsolved conjecture was provided. 

%
%

The  rest of the paper
is organized as  follows.
In Section  \ref{pre},
Some preliminaries related to the subject are provided,  including 
the  definition  for  
tensor  eigenpairs  and  robust  
eigenpairs,  and  the  focused regular simplex tensor. 
In  Section  \ref{Relationonrobust},
the relationship 
between  the   robust    and locally maximized eigenpairs  
of  tensor 
was  investigated. 
 Then, in Section  \ref{NewCriterion},
 the proof  for the conjecture was  provided.
 Some future works are discussed in Section  \ref{futurework}.

\section{Preliminaries}\label{pre}
We    introduce  some  necessary   notations, definitions   and  lemma used  in   this  article. 
In  this  material,  as adopted in  many tensor-related works \cite{kolda,TensorPCA,9521829},  high-order tensors are denoted
in
boldface Euler script letters, e.g.,
$\mathcal  A $.
Matrices are denoted  in   boldface capital letters, e.g., $\mathbf  A $; vectors are denoted in  
boldface lowercase letters, e.g.,  $\mathbf  a $.
Sets and subsets are denoted  in blackboard bold  capital letters, e.g.,  $\mathbb  A $.

A  $d$th-order tensor is denoted 
$\mathcal A \in \mathbb {R}^{I_1 \times I_2  \times \dots \times I_{d} }$,
where 
$d$
is the order   of
$\mathcal A $, and 
$ I_j $ ($  j \in \{ 1,2,\dots,d \}$)  is  the  dimension  of  
$j$th-mode.
The element of $\mathcal A$,    which  is  indexed  by integer tuples $(i_1,i_2,\dots,i_d) $, is denoted 
($a_{i_1,i_2,\dots,i_d})_
{1 \le i_1 \le I_1, 
	\dots, 
	1 \le i_d \le I_d}
$. 
A   tensor is called    symmetric if its elements remain invariant under any permutation of 
the  indices\cite{kolda}. 
Let    $  T^{m}(\mathbb R^{n}) $ denotes   the  space  of  all  such  real  symmetric    tensors.
Given a $m$th-order  $n$-dimensional symmetric  tensor  $\mathcal S $ and  a  vector $ \mathbf u \in \mathbb {R}^{n \times 1}$, we have
$ 
\mathcal S \mathbf u^{m} =
\sum\limits_{i_1,i_2,\dots,i_m=1}^{n} 
s_{i_1,i_2,\dots,i_m}  u_{i_1} \dots   u_{i_{m}}
$, 
and  $   \mathcal S \mathbf u^{m-1}   $   denotes   a    $n$-dimensional
column   
vector,  whose  $j$th  element   is    
$
(\mathcal S \mathbf u^{m-1})_{j} =
\sum\limits_{i_2,\dots,i_m =1}^{n} 
s_{j,i_2,\dots,i_m}  u_{i_2} \dots   u_{i_{m}}
$\cite{Cuicf}.
Furthermore,   $   \mathcal S \mathbf u^{m-2}   $   is  an    $n  \times  n $  matrix,    whose  $(i, j )$th  element   is  
$
(\mathcal S \mathbf u^{m-2})_{i, j} =
\sum\limits_{i_3,\dots,i_m =1}^{n} 
s_{i,j,i_3,\dots,i_m}  u_{i_3} \dots   u_{i_{m}}.
$

$\mathbf  1_{n}$   denotes   a  $ n  \times  1$ column vector,  and  $\mathbf  I_{n}$   denotes   a  $ n  \times  n$  identity  matrix.

$\triangledown$ is  the  gradient  operator.
$ \rm null( \mathbf A) $ 
denotes the null space  of $\mathbf A$. 
$ \mathbf  P_{\mathbf  A }^{\bot} $ is  the  orthogonal   cpmplement  operator  of  $\mathbf A$.

\begin{definition}[\textbf{Spectral  radius}]\label{radiusdefinition}
	The  spectral  radius  of    a  matrix  $\mathbf A \in \mathbb {R}^{n \times n}$  is  the   maximum  value of  the  absolute  value  
	of  all    eigenvalues  of 
	$\mathbf A$,  denoted  by 
	$ \rho (\mathbf A) = max  \vert  \sigma_{i}(\mathbf A) \vert$,
	where 
	$ \sigma_{i}(\mathbf A)  ,  i=1,2,\dots, n $
	are   $n$ eigenvalues of 
	$\mathbf A$.
\end{definition}

\begin{definition}[\textbf{The  outer product}]
	\label{outerprod}
	Given   $m$  vectors 
	$ \mathbf a^{(i) } \in \mathbb {R}^{I_i \times 1}$ 
	($i=1,2, \dots, m$),
	their   outer  product   
	$ \mathbf a^{(1) }
	\circ
	\mathbf a^{(2) }
	\circ  \dots
	\circ
	\mathbf a^{(m) } 
	$  
	is   a    $ m$th-order  tensor denoted $   \mathcal A$,  with  a size  of  
	$ I_1 \times I_2  \times \dots \times I_m  $.  
	And  its    element is    the  product  of    the  corresponding  vectors'   elements, i.e., 
	$ 
	a_{i_1,i_2,\dots,i_m}
	= 
	\mathbf a^{(1) }_{i_1}
	\mathbf a^{(2) }_{i_2}
	\dots
	\mathbf a^{(d) }_{i_m} 
	.
	$
	When 
	$  \mathbf a^{(1) }
	=  
	\mathbf a^{(2) }
	=   \dots
	= 
	\mathbf a^{(m) }
	=\mathbf a $ $(I_1 =  I_2  = \dots = I_m =I)$,  we use  the  notation 
	$ \mathcal A =  \mathbf a^{\circ m}$  for  simplicity,  where 
	$   \mathcal A $ is  a  symmetric  tensor  of  order  $m$  and  dimension  $ I$.  
\end{definition}


\subsection{Optimization  theories of  Tensor eigenpairs}
In this part,  we briefly  introduce the  optimization theories   related to the tensor  eigenpairs problem. 
The  concept  of  tensor  eigenpairs  can be    understood  and  derived  by  considering  the  following   constrained  optimization  model:
\begin{equation}\label{opti_ori}
\begin{cases}
\max\limits_{\mathbf v} \quad \mathcal S \mathbf v^{m}   \\
\rm s.t. \quad \mathbf v^{\mathrm {T}}\mathbf v=1
\end{cases}.
\end{equation}
The Lagrangian function 
of (\ref{opti_ori})  is  defined as:
\begin{equation}\label{Lagrangian_function}
L(\mathbf v, \lambda)=
\frac {1}  {\it m}
\mathcal S \mathbf v^{\it m}
+
\frac { \lambda} {2} (1-  \mathbf v^{\mathrm {T}}\mathbf v).
\end{equation}
When the gradient of
$  L(\mathbf v, \lambda) $ to
$ \mathbf  v $ (also known as  ) is  $ \mathbf 0$,    the eigenpair of a  symmetric  tensor can  be  deduced,  which  was  independently    defined  by   Lim  and  Qi  in  2005:

\begin{definition} \cite{qi,lim} [\textbf{eigenpairs of  symmetric tensor}]
	Given a  tensor $\mathcal S   \in    T^{m}(\mathbb R^{n}) $,
	a pair
	$(\lambda ,\mathbf v )$
	is an  eigenpair  of  
	$\mathcal S  $ 
	if
	\begin{equation}\label{definition}
	\mathcal S \mathbf v^{m-1}=\lambda \mathbf v,
	\end{equation}
	where
	$ \lambda  \in  \mathbb C $
	is  the  eigenvalue and
	$ \mathbf v  \in   \mathbb C^{n \times  1} $
	is the  corresponding   eigenvector   satisifying 
	$\mathbf v^{\mathrm {T}}\mathbf v=1 $.
\end{definition}

The  second-order  derivation  information  plays  an  important  role  in  identifying    whether  a  stationary  point  is     locally    optimal    given    an  optimization   model.      The    second-order derivation  
of
$  L(\mathbf v, \lambda) $ 
to    $ \mathbf v $, which is  also  termed  the Hessian matrix of
(\ref{Lagrangian_function}),  is   denoted
by 
\begin{equation}\label{hessian_matrix}
\mathbf H(\mathbf v) = (m-1)\mathcal S \mathbf v^{m-2} - \lambda \mathbf I_{n} ,
\end{equation}
where
$ \mathbf I_{n} $
is an $ n  \times  n $  identity matrix.

Then,  
the 
locally    optimal   solutions of   (\ref{opti_ori})
can  be  identified  by
checking the  negative 
definiteness  
and determining the sign of each  eigenvalue
of the  following  matrix:
\begin{align}\label{Mhess}
\mathbf {K}
& =  \mathbf  P_{\mathbf  v}^{\bot} \mathbf H (\mathbf v)   \mathbf  P_{\mathbf  v}^{\bot}
\end{align}
and the  detailed  explanation for  the  
derivation of 
(\ref{Mhess})
can refer to  
our previous work \cite{localRST}.

	\subsection{Robust  eigenpairs of  symmetric tensors}\label{robustsection}

One of  the  widely used  algorithms  to  obtain  tensor   eigenpairs  is  called  the  tensor  power   method, which 
is  based on  the  following  mapping  function   
\begin{equation}
\phi(\mathbf{v})=\frac{\mathcal S  \mathbf{v}^{m-1}}{\left\|\mathcal S  \mathbf{v}^{m-1}\right\|}
\end{equation}
and  performs  the  following  iterative  schedule:	
\begin{equation}\label{tenpower}
\mathbf  {v}_{k+1} \mapsto \frac{\mathcal{S}  \mathbf{v}_{k}^{m-1}}{\left\|\mathcal{S}  \mathbf{v}_{k}^{m-1}\right\|}
.
\end{equation}
It can be observed 
by comparing  
(\ref{tenpower})
with the KKT condition (\ref{definition})
that the tensor power method 
can be understood as iterative method by 
using the  first-order KKT  equation. 

Given  a random   initialization  vector  
$ 	\mathbf  {v}_{0} $
and  iterate the above  formula 
(\ref{tenpower}) until  some  termination condition  is  satisfied,  an  eigenpair  can be  obtained. 
Then,  a   robust eigenvector of $ \mathcal{S} $  is an eigenvector  $  \mathbf v $ that is an attracting
fixed point of the tensor power method, 
which 
can be identified  by  calculating  the spectral radius of the Jacobian  matirx   of  the  map  function,  which  is   denoted  and  given  by 	
\begin{equation}\label{Jacobianmatirx} 
\mathbf{J}(\mathbf{v})=\frac{m-1}{\lambda}\left(\mathcal{S}  \mathbf{v}^{m-2}- \lambda  \mathbf{v} \mathbf{v}^{\top}\right),
\end{equation}
where $  \lambda \neq 0$, indicating that 
the robustness  checking should exclude these eigenpairs  with eigenvalues 0.
If the  spectral  radius  of  $  	\mathbf{J}(\mathbf{v}) $ is  less than  1,  it  is  said  that  the  corresponding  eigenpair  is  a  robust  one. 
The  detailed derivations  can  refer  to 
the  Lemmas 3.2 and 3.3  in \cite{RobustEigen}.

\begin{remark}
	It should be noted that
	the proof for the 
	the  Lemma 3.3  in \cite{RobustEigen}
	is   less  rigorous or 
	is correct but miss
	some necessary illustration for the final expression 
	of (\ref{Jacobianmatirx}).

	Strictly speaking,  according to the derivation presented in the proof of Theorem 3.3 in \cite{RobustEigen},
	one can only obtain 
	the Jacobian matrix  as  follows:
	\begin{equation}\label{Jacobianmatirxabs} 
	\mathbf{J}(\mathbf{v})=\frac{d-1}{\vert  \lambda \vert }\left(\mathcal{S}  \mathbf{v}^{d-2}- \lambda  \mathbf{v} \mathbf{v}^{\top}\right), 
	\end{equation}
	which is  used  in the  sequential work  that  discussed  the same problem \cite{teneigenstructure}.
	The minor difference lies in that
	one more absolute value operation 
	was imposed on  the numerator $  \lambda  $.
	Note that even though 
	such a difference 
	will have no impact on determining
	the spectral radius, 
	since  we still need to calculate the absolute value one more time when determining  the spectral radius, as can be seen in the  definition \ref{radiusdefinition}.
	In this sense, 
	the difference between 
	(\ref{Jacobianmatirx})
	and 
	(\ref{Jacobianmatirxabs})
	is trivial.

	Here, we still would like to 
	justify 
	that
	(\ref{Jacobianmatirx})
	is also correct  to directly use   the original 
	(\ref{Jacobianmatirx}).
	The reasons  are as  follows. 
	When  $  d $ is odd, both  
	$(\lambda,\mathbf v )$   and   $(-\lambda ,-\mathbf v )$   are    eigenpairs,  without  losing 	generality,  we     only  take   eigenpairs with   non-negative   eigenvalues;  when  $  d$ is even,    eigenpairs   appear  in  $ \pm$ pairs  with  the same  eigenvalue, we only select $\mathbf v$ satisfying $ \sum_{i}v_{i}>0$. 
	In  addition, 
	we can derive that 
	$ 
	\lambda=
	\mathcal{S}  \mathbf{v}^{m}
	=
	\sum_{i=1}^{n+1} ( \mathbf  {w}_{j}^{\mathrm T} \mathbf{w}_{i})^{ m} 
	>0$
	for even $m$.

\end{remark}

\subsection{Regular  simplex  frame  and  tensor}\label{conjecture}
In this part,  we  will  introduce  a  special class of symmetric   tensors,  which is termed   
regular  simplex  one. 
First,  the    definition of the  generalized  equiangular set  is  introduced as  follows:  
\begin{definition}[Definition 4.1  in Ref. \cite{RobustEigen}]
	An equiangular set (ES) is a collection of vectors $\mathbf{w}_{1}, \ldots, \mathbf{w}_{r} \in \mathbb{R}^{n}$ with $r \geq n$ if there exists $\alpha \in \mathbb{R}$ such that
	$$
	\alpha=
	\left|\left\langle\mathbf{w}_{i}, \mathbf{w}_{j}\right\rangle\right|, \forall i \neq j \quad \text { and } \quad\left\|\mathbf{w}_{i}\right\|=1, \forall i .
	$$
	Furthermore, an $\mathrm{ES}$ is  called  an equiangular tight frame (ETF)  if  
	\begin{equation}
	\mathbf{W} \mathbf{W}^{\mathrm T}=a \mathbf{I}, \quad \mathbf{W}:=\left(\mathbf{w}_{1},  \cdots,  \mathbf{w}_{r}\right) \in \mathbb{R}^{(n-1) \times r}
	\end{equation}
\end{definition}

For  example, 
when   $r=n$ and $a=1$, the  orthonormal bases
$\left\{\mathbf{w}_{1}, \ldots, \mathbf{w}_{n}\right\} \subset \mathbb{R}^{n \times n}$
forms  an  ETF,  where $\alpha =0$.
When    $r=n+1$ and $a=\frac {n+1}{n}$,   $\left\{\mathbf{w}_{1}, \ldots, \mathbf{w}_{n+1}\right\} \subset \mathbb{R}^{(n) \times  (n+1)}$, it  is   termed  
regular  simplex frames,    $ 
\alpha = - \frac{1}{n}$.
$\mathbf{w}_{i}, i=1,2 \dots, n+1$   are  called  the  vectors  in  the  frame. 
In  2D  space,  
the  regular simplex  frame  is  a   regular  triangle.

%

Then,  the  regular  simplex  tensor  is   one  deduced  by  the  the  regular  simplex   frame  with  the  following   form:
\begin{equation}\label{simplextensor}
\mathcal{S}:=\sum_{i=1}^{n+1} \mathbf{w}_{i}^{\circ m}
\end{equation}
where $\circ$ is the outer product  defined  in Definition \ref{outerprod}, 
and  	$\mathcal{S}
$ is  a   symmetric  tensor of order
$ m $
and dimension
$ n$.   
In  this  paper,  we   mainly  analyze  the  robust  eigenpairs  of  regular  simplex  tensor  and  try  to  deal  with  the  following  conjecture:  
\begin{conjecture}[Conjecture 4.7 in  Ref  \cite{RobustEigen} ] \label{conjecturesimplex}
	The robust eigenvectors of a regular simplex tensor   are precisely the
	vectors in the frame.
\end{conjecture}	

This  conjecture 
was originally  proposed in \cite{RobustEigen}.
Later, several researchers  also  considered this problem  in \cite{teneigenstructure}. 
It should be noted that the above conjecture contains 
the following meanings:
1)   the  robust eigenpairs of   regular simplex tensors  may not exist, which has been observed by experiments in  \cite{RobustEigen}, 
and will also be theoretically justified in the later  analysis;
2) if exists, they are only these 	vectors in the frame.
In \cite{RobustEigen}, the authors only shows that 
these vectors in the frame are indeed the robust one 
for the tensor  with dimension $n$ and order $m$ and $n \ge 2$,  $m \ge 3$ and $ n+m \ge 7$. (See Theorem 4.5 of \cite{RobustEigen} for details).
However, 
the uniqueness that these vectors in the frame are only robust ones is not justified, which was then left as the above Conjecture
\ref{conjecturesimplex}. 
In other words, the term "precisely" in the conjecture remains to be demonstrated. 
For this purpose, 
one must check  the robustness of  all  eigenvectors given  a regular simplex tensor,
and identify which are  robust ones  among all  candidates.  
And such a strategy  has been adopted  in the later work  \cite{teneigenstructure}.
In detail, they first analyzed   the real eigenstructure of  all eigenpairs   concerning   regular simplex tensor, and 
determine 
whether each eigenpair is robust or not  by  checking the spectral radius of the Jacobian matrix at each one. 
However, it can be seen that such a task is computationally heavy as $n$ and $m$ become larger, since the number of all eigenpairs 
will exponentially increase.  
So, they   only provided   the proof for the simplest case $n=2$, 
while the  higher $n$ cases will be difficult to be calculated and checked. 
In addition, 
it is also difficult 
for   most of eigenpairs 
to 
implicitly calculate the Jacobian 
matrix, let alone 
determine the spectral radius further.

To  our  best  knowledge,  so far, 
these  references are the only two works 
focusing on  identifying  the  robust  eigenpairs of regular  simplex  tensor. 
In this paper,  in the basis of the two works,  
we proceed further concerning the above  conjecture and 
finally complete the proof for the above  conjecture. 
The details are as follows. 
%

 \section{Relation on   robust and locally maximized eigenpairs }\label{Relationonrobust}


Note  that  in  the  related works \cite{RobustEigen,teneigenstructure},
the  robustness of  an  eigenpair  is  generally  checked  by the  spectral  radius of the Jacobian  matirx as  defined  in 
Subsection   \ref{robustsection}.
However, as mentioned above, 
it is a little complicated 
to determine the spectral radius of all eigenpairs, especially for the higher $n$ and $m$    case.

In this part,  we  would  like  to  provide  an  auxiliary   criterion   for  robustness  checking. 
It  can  be  observed   by   comparing    (\ref{Jacobianmatirx})  with   (\ref{hessian_matrix})  
that  both  of  them   contain   a  term  $\mathcal{S}  \mathbf{v}^{m-2}$,  and  
by  further   investigating  their  relationship,  the  following    lemma  can  be  built:     
\begin{lemma}\label{RobustLocal}
	Given  an  eigenpair   
	$(\lambda ,\mathbf v )$  of   a  tensor $\mathcal S   \in    T^{m}(\mathbb R^{n}) $, 
	if  the eigenpair  is  a  robust  one,  
	it must be  locally  maximized    solution  of  model (\ref{opti_ori}).
	Conversely, it does not always hold. 	
\end{lemma}

\begin{proof}
	First,  
	$ 	\mathbf {K} $   in  (\ref{Mhess})  is  rewritten  as  the  following  form:
	\begin{align}
	\mathbf {K}
	&   	  	\nonumber 
	=  \mathbf  P_{\mathbf  v}^{\bot} \mathbf H (\mathbf v) 
	\mathbf  P_{\mathbf  v}^{\bot} 
	\\  	\nonumber 
	&=
	(\mathbf  I_{n}  - 
	\mathbf v\mathbf  v^{\mathrm T} )
	[ (m-1)\mathcal S \mathbf v^{m-2} - \lambda \mathbf I_{n}]
	(\mathbf  I_{n}  - 
	\mathbf v  \mathbf v^{\mathrm T} )
	\\  	\nonumber 
	&=
	(\mathbf  I_{n}  - 
	\mathbf v\mathbf v^{\mathrm T} )
	[
	(m-1)\mathcal S \mathbf v^{m-2} 
	-
	\lambda \mathbf I_{n}
	-
	(m-2)\mathbf v\mathbf v^{\mathrm T}  
	]
	\\  	\nonumber 	
	&=
	(m-1)\mathcal S \mathbf v^{m-2} 
	-
	\lambda \mathbf I_{n}
	-
	(m-2)\mathbf v\mathbf v^{\mathrm T}  
	-
	(m-1)\mathcal S \mathbf v^{m-2}  	\mathbf v\mathbf v^{\mathrm T}
	+
	\lambda	\mathbf v\mathbf v^{\mathrm T}
	+
	(m-2)	\mathbf v\mathbf v^{\mathrm T}
	\\    	 	\nonumber 
	&
	=
	(m-1)\mathcal S \mathbf v^{m-2} 
	-
	\lambda \mathbf I_{n}
	-
	(m-1)\lambda \mathbf v \mathbf v^{\mathrm T}
	+
	\lambda	\mathbf v\mathbf v^{\mathrm T}  ,	
	\end{align}
	where in the last  equation, we utilize 
	$ \mathcal S \mathbf v^{m-2}  	\mathbf v 
	= \mathcal S \mathbf v^{m-1}  	
	=
	\lambda		\mathbf v$.	
	By   comparing  with   (\ref{Jacobianmatirx}), it  holds  that  	
	\begin{align}\label{lmdJMI}
	\lambda	\mathbf J =   	{\mathbf K }
	+\lambda  (\mathbf  I_{n}  - 
	\mathbf v \mathbf v^{\mathrm T} ) .
	\end{align}
	
	Furthermore,  it can be checked that 
	\begin{equation}\label{Jacobianmatirxvec} 
	\mathbf{J} \mathbf{v}
	=
	\frac{m-1}{\lambda}
	(
	\mathcal{S}  \mathbf{v}^{m-2}\mathbf{v} 
	- \lambda  \mathbf{v} \mathbf{v}^{\mathrm T} \mathbf{v})
	=
	\frac{m-1}{\lambda}
	(
	\lambda  \mathbf{v} 
	- \lambda  \mathbf{v})
	=
	0 \cdot 
	\mathbf{v},
	\end{equation}
		where   we  reuse 
	$ \mathcal S \mathbf v^{m-2}  	\mathbf v 
	= \mathcal S \mathbf v^{m-1}  	
	=
	\lambda		\mathbf v$, and the above result 	 means 
	that
	$  	\mathbf{v}$
	is an eigenvector of 
	$ \mathbf{J} $ with eigenvalue 0,
	and in this way, the other eigenvectors must lie in the 
	null  space of 	$  	\mathbf{v}$, denoted by 
	$null(\mathbf{v})$.

	In  a  	similar way,  
	the  vector  $\mathbf v$ 
	will  simultaneously 
	be the  eigenvector  for  
	$ \mathbf K,  \mathbf  I_{n}  - 
	\mathbf v  \mathbf v^{\mathrm T}$, 
	which  holds
	that 
	\begin{equation}\label{Kmatirxeig} 
	\mathbf{K} \mathbf{v} =
	0 \cdot 
	\mathbf{v}, 
	\quad 
	(\mathbf  I_{n}  - 
	\mathbf v  \mathbf v^{\mathrm T})\mathbf{v} =
	0 \cdot 
	\mathbf{v}, 
	\end{equation} 	
	and  the  other  eigenvectors    
	of 	$ \mathbf K,  \mathbf  I_{n}  - 
	\mathbf v  \mathbf v^{\mathrm T}$  also  must  lie  in   		$null(\mathbf{v})$.
	
	This  indicates that  their  eigendecomposition 
	shares   the  same  eigenvector  matrix,  denoted by 
	$\mathbf V
	=
	[\mathbf v,   		null(\mathbf{v})] \in 
	\mathbb R^{ n  \times  n} $.
	In  other word, concerning  (\ref{lmdJMI}),  
	they can be  diagonalized  
	by the same  orthogonal matrix  	$\mathbf V$,  which  follows:
	\begin{align}\label{lmdJMIVdiag}
	\lambda  \mathbf V^{\mathrm T}   
	\mathbf J  
	\mathbf V=   
	\mathbf V^{\mathrm T}
	[ 	{\mathbf K }
	+\lambda  (\mathbf  I_{n}  - 
	\mathbf v \mathbf v^{\mathrm T} )
	]
	\mathbf V.
	\end{align}
	Note  that  
	$  
	\mathbf  I_{n}  - 
	\mathbf v  \mathbf v^{\mathrm T}
	\in 
	\mathbb R^{ n  \times  n}
	$ 
	is  a  projection  matrix  with  rank  $n-1$, and  its  eigenvalues are given by 
	$ 1, 1,   \dots,  1,  0    $,  where the  number  of  the eigenvalue  of  $ 1 $  is  $ n-1 $.  
	
	Based on the  above relationship (\ref{lmdJMIVdiag}), once 
	we obtain the  eigen-distribution of 
	$ \mathbf K $,
	the corresponding eigenvalues of 
	$ \mathbf J$
	and its spectral radius can be determined. 
	Therefore, in the following, we will discuss that under what condition, 
	the eigenpair may be a robust one. 
	
	Assume that 
	the 
	$ n $  eigenvalues of  $\mathbf J$
	and  $\mathbf K$
	is sorted  as 
	\begin{align}
	\sigma_{1}^{\mathbf J}   \le   \sigma_{2}^{\mathbf J}   \le    \dots  \le \sigma_{n}^{\mathbf J} .
	\\
	\sigma_{1}^{\mathbf K}   \le   \sigma_{2}^{\mathbf K}   \le    \dots  \le \sigma_{n}^{\mathbf K} .
	\end{align}

	  Then, we can have that 
	  	\begin{align}\label{lmdrelationequ}
	  	\lambda	 
	  	\sigma_{i}^ {\mathbf J} 	
	  	=
	  	\begin{cases}
	  		\sigma_{i}^ {\mathbf K} 
	  	+		\lambda    	   ,
	  	\quad     \sigma_{i}^{\mathbf  I_{n}  - 
	  		\mathbf v \mathbf v^{\mathrm T}} =1 
	  	\\
	  	\sigma_{i}^ {\mathbf K}    =0 	  , 
	  	 \quad     \sigma_{i}^{\mathbf  I_{n}  - 
	  		\mathbf v \mathbf v^{\mathrm T}} =0 
	  	\end{cases}.
	  	\end{align}

	
	According to the sign of all  eigenvalues of 	$  	\mathbf K  $, 
	three different cases are discussed as  follows: 
	
	\textbf{Case 1}:
	If  	
	$(\lambda ,\mathbf v )$  is    a   locally  minimized     solution  of  model (\ref{opti_ori}), 
	its  corresponding   matrix 
	$  	\mathbf K  $ 
	will be  a  non-negative  definite one.
	By combining 
	(\ref{Kmatirxeig}), 
	it can be concluded that 
	$   \sigma_{min}^{\mathbf K}=	\sigma_{1}^{\mathbf K}  =0 $,
	and 
	$ 0 \le   \sigma_{2}^{\mathbf K}   \le    \dots  \le \sigma_{n-1}^{\mathbf K}  $.
	By  combining  (\ref{lmdJMIVdiag}), it can be  concluded that 
	\begin{align}\label{lmdrelationcase1}
	\sigma_{min}
^{\mathbf J } =	\sigma_{1}
	^{\mathbf J }=0,
	\quad  
	\lambda 	\sigma_{i}
	^{\mathbf J }	
	=    	
	\sigma_{i} 	^{\mathbf K}	
	+	\lambda  
	\ge 
	\lambda  	,  
	\quad  i= 2,\dots, n. 
	\end{align}
	which  indicates 
	all eigenvalues of 
	$ \mathbf J $
	are non-negative, and 
	it can be derived that 
$ 	\rho (\mathbf J)
	=
		\sigma_{n}
	^{\mathbf J }
\ge 
	1 $, 
	which indicate that the corresponding  eigenpair  cannot be a  robust one.
	
	\textbf{Case 2}:
	If  	
	$(\lambda ,\mathbf v )$  is    a   saddle   solution  of  model (\ref{opti_ori}), 
	its  corresponding   matrix 
	$  	\mathbf K  $ 
	will be  uncertain,  indicating 
	its eigenvalues include negative, positive, and 0.
	Assuming that 
	$k$-th eigenvalue are equal to 0, which follows:
	\begin{align}
	\sigma_{1}^{\mathbf K}   \le  \dots  \le   \sigma_{k-1}^{\mathbf K}  \le  \sigma_{k}^{\mathbf K} =0   \le     \sigma_{k+1}^{\mathbf K}  \dots  \le \sigma_{n-1}^{\mathbf K} .
	\end{align}
	Since 
	(\ref{Kmatirxeig})  shows  that 
	$\mathbf J$ 
	is always with one eigenvalue 0, 
	based on   (\ref{lmdJMIVdiag}), 
	it  indicates  that
	$ \sigma_{k}^{\mathbf J} =0 $, 
	and 
	$ \sigma_{k}^{\mathbf  I_{n}  - 
		\mathbf v \mathbf v^{\mathrm T}} =0 $. 
This further means that except for the $k$-th eigenvalues,
the other ones  of 
	$\lambda  (\mathbf  I_{n-1}  - 
	\mathbf v \mathbf v^{\mathrm T})$ are all 
	$    \lambda   $,
Then, we have that 
	\begin{align}\label{lmdrelationcase2}
		\lambda    \sigma_{i}
	^{\mathbf J }	
	&=    	
	\sigma_{i} 	^{\mathbf K}	
	+	\lambda  
	\le 
	\lambda  	,  
	\quad  i= 1,\dots, k-1. 
	\\ 
	\lambda    \sigma_{i}
	^{\mathbf J }	
	&=    	
	\sigma_{i} 	^{\mathbf K}	
	+	\lambda  
	\ge 
	\lambda  	,  
	\quad  i= k+1,\dots, n-1. 
	\end{align}

It can be seen that for 	$ i=  k+1,\dots, n-1 $, the corresponding 
eigenvalues  	$ \lambda    \sigma_{i}
^{\mathbf J } $	have definitely larger than $\lambda$. 
In this sense, 
no matter  	the range concerning  $\lambda    \sigma_{i}
^{\mathbf J }$ for 
	$  i=1,\dots, k-1 $, 
it  always holds that 
$
	\rho (\mathbf J)
	>1$
	and the   corresponding  eigenpair  cannot be a  robust one.

	\textbf{Case 3}:
	If  	
	$(\lambda ,\mathbf v )$  is    a   locally  maximized    solution  of  model (\ref{opti_ori}), 
	its  corresponding   matrix 
	$  	\mathbf K  $ 
	will be  a  negative  definite one,  indicating 
	$   \sigma_{max}^{\mathbf K}=	\sigma_{n}^{\mathbf K}  =   0 $
	and 
	$    \sigma_{1}^{\mathbf K}   \le    \dots  \le \sigma_{n-1}^{\mathbf K} \le 0 $.
	By  combining  (\ref{lmdJMIVdiag}), it can be  concluded that 
	\begin{align}\label{lmdrelationmax}
	\sigma_{max}
^{\mathbf J } =	\sigma_{n-1}
^{\mathbf J }=0,
\quad  
\lambda 	\sigma_{i}
^{\mathbf J }	
=    	
\sigma_{i} 	^{\mathbf K}	
+	\lambda  
\le 
\lambda  	,  
\quad  i= 1,\dots, n-2.  
	\end{align}
In this case, 
it can be observed that 
all eigenvalues have been bounded 
which is less than $ \lambda$. 
Furthermore, if \begin{align}\label{add}
 \lambda 	\sigma_{i}
^{\mathbf J }	
=    	
\sigma_{i} 	^{\mathbf K}	
+	\lambda  
>  
- \lambda  	,  
 i= 1,\dots, n-2.  
 \
 \end{align} 
and  (\ref{lmdrelationmax}) cannot 
hold for the strict equality, we can conclude that 
 all eigenvalues 
 lie in the range 
 $(-\lambda, \lambda)$,
	which  indicates that
$ 	\rho (\mathbf J)
<
1 $, 
	
	The above three cases discussion 
	indicated that 
	if  only  
and  	if  		$(\lambda ,\mathbf v )$  is    a   locally  maximized    solution  of  model (\ref{opti_ori}), 
with additional  condition as shown in (\ref{add}), 
it could be a robust one. 
While  
the locally minimized and saddle solutions must not be robust. 

	
	Conversely, if  an eigenpair is a robust one
	with 	 	 $
\rho  (\mathbf J ) 	
	< 1.
	$,
	This 
	equivalently indicates that 
	all eigenvalues are bounded  
	in the range of 
	$(-\lambda, \lambda)$. 
Similarly  	by     (\ref{lmdrelationequ}), we have that 
	$	\sigma_{i} 	^{\mathbf K} < 0$, which means that 	
		$(\lambda ,\mathbf v )$  is    a   locally  maximized    solution  of  model (\ref{opti_ori}).
			And  the proof is complete.
\end{proof}
	
	To  conclude, 
	if an eigenpair is a robust one, 
	it must be a locally  maximized   eigenpair  of the corresponding  constrained model  (\ref{opti_ori}).
	Conversely, it does not always hold.
	Therefore, the  set of all robust eigenpairs 
	is a subset of these locally maximized eigenpairs,
	which will be  always  no large than that of all eigenpairs. 
	In other words, 
	only these locally maximized eigenpairs could be the potential 
	candidates 
	to be  the  robust eigenpairs. 
	
	In this way, 
	given a  tensor (here, it is not constrained to be a  regular simplex tesnor since the above lemma is established for any symmetric tensors), 
	we do not need to analyze and determine the spectral radius 
	of $\mathbf J$   at  all eigenpairs, as adopted in the previous  work.  
	In contrast, 
	it suffices to 
	only focus on these locally maximized 
	eigenpairs  of  tensor.

	In this sense, 
	such a lemma 
	will bring about two advantages concerning robust eigenpairs identification for the Conjecture \ref{conjecturesimplex}:
	
	1) the number of eigenpairs to be  checked 
	is  greatly  reduced, due to the fact that 
	the number of   these locally maximized 
	eigenpairs is always no larger  than that of 
	all eigenpairs,  especially when $n$ becomes  larger. 
	
	2):
	naturally,  following the first merit, the difficulties   existed in the previous work 
	that   
	calculating the Jacobian matrix and  determining the spectral radius  for most of eigenpairs
	can be avoided, 
	since here we only focus on these  locally maximized 
	eigenpairs (these vectors in the frame), which is easy to be calculated, as can be seen  from 
	the following part.

  	Therefore, 
  	the following aim   turns  to
  	identifying   which are 
  	locally maximized solutions, concerning regular simplex tensors, which could be another important issue. 
  	Such an  
  	issue has been systematically investigated in our previous work\cite{localRST}  and 
  	we have
  	theoretically proved  that 
  	these locally maximized solutions of  a regular simplex  tensor
  	are 
  	only 
  	these vectors in the  regular simplex frame.
  	Please refer to Lemma 1 to 5   in   \cite{localRST} for details.

  	In this sense, we have only 
  	check the robustness of these vectors in the frame. 
We finish the proof for the Conjecture
\ref{conjecturesimplex}
in the next section.

 \section{Proof of the Conjecture \ref{conjecturesimplex}}\label{NewCriterion}
%
 
 \setcounter{conjecture}{0} 
\begin{conjecture}[Conjecture 4.7 in  Ref  \cite{RobustEigen} ] \label{conjecturesimplex}
	The robust eigenvectors of a regular simplex tensor   are precisely the
	vectors in the frame.
\end{conjecture} 

\begin{proof} 
Concerning  the   vectors   in the regular simplex  frame
$
\mathbf W =
[
\mathbf{w}_{1}, \ldots, \mathbf{w}_{n}
] \in
\mathbb{R}^{(n-1) \times  n}$, 
it holds that 
$ \mathbf{w}_{i}^{\mathrm T} \mathbf{w}_{i} =1, 
\mathbf{w}_{i}^{\mathrm T} \mathbf{w}_{j} = \alpha = -\frac{1}{n} $
for 
$ i \neq j$. 
$	\mathbf{W} \mathbf{W}^{\mathrm T}=
\sum_{i=1}^{n} 
\mathbf{w}_{i}
\mathbf{w}_{i}^{\mathrm T}
=
\frac{n+1}{n} \mathbf{I}$. 
It can be calculated that  when 
$ \mathbf v = \mathbf w_{j}$,
it holds that 

\begin{align}\label{swj}
\mathcal{S}  \mathbf v^{m-2}    
& =
\mathcal{S}  \mathbf w_{j}^{m-2}  
=
(
\sum_{i=1}^{n+1} \mathbf{w}_{i}^{\circ m}
)  \mathbf w_{j}^{m-2} 
=
\sum_{i=1}^{n+1} (  \mathbf w_{j}^{\mathrm T} \mathbf{w}_{i})^{ m-2} 
\mathbf{w}_{i}
\mathbf{w}_{i}^{\mathrm T}
\nonumber   \\
&= 
\mathbf{w}_{j}
\mathbf{w}_{j}^{\mathrm T}
+
\frac
{  1}
{ (-n)^{ m-2}  } 
\sum_{i=1, i \neq j}^{n+1} 
\mathbf{w}_{i}
\mathbf{w}_{i}^{\mathrm T}
=
\mathbf{w}_{j}
\mathbf{w}_{j}^{\mathrm T}
+
\frac
{  1}
{ (-n)^{ m-2}  }   
(
\frac{n+1}{n} \mathbf{I}
-
\mathbf{w}_{j}
\mathbf{w}_{j}^{\mathrm T}
)
\nonumber   \\
&=
(
1
-
\frac
{  1}
{ (-n)^{ m-2}  }  
)
\mathbf{w}_{j}
\mathbf{w}_{j}^{\mathrm T}
+
\frac
{ n+1}
{ (-n)^{ m-2} n }   \mathbf{I}
.
\end{align}

In  addition, 
the  corresponding  eigenvalue can be calculated by 
\begin{equation}\label{lmdwj}
\lambda
=
\mathcal{S}  \mathbf  {w}_{j}^{m}    
=
\sum_{i=1}^{n+1} ( \mathbf  {w}_{j}^{\mathrm T} \mathbf{w}_{i})^{ m} 
=
1+
\frac
{  n}
{ (-n)^{ m}  }   
.
\end{equation}

Then,  according to  (\ref{Jacobianmatirx}), 
 the Jacobian matrix  
at the eigenvector  $ \mathbf  {w}_{j} $
will satisfy 
\begin{equation}
 \mathbf  J (\mathbf  {w}_{j} )
 \mathbf  {w}_{j}
 =
 \frac{m-1}{\lambda}
(
\mathcal{S}  \mathbf  {w}_{j}^{m-2}- \lambda  \mathbf  {w}_{j} \mathbf  {w}_{j}^{ \mathrm T }
)
 \mathbf  {w}_{j}
 =
 0
 \cdot
  \mathbf  {w}_{j}
\end{equation}
	which means 
that
$  {\mathbf  {w}_{j} }$ itself  
is an eigenvector of 
$  \mathbf  J (\mathbf  {w}_{j} )  $ with eigenvalue 0,
and naturally, the other eigenvectors must lie in the 
null  space of 	$  	\mathbf {w}_j$, denoted by 
$null(\mathbf  {w}_{j}  )$.
By further 
utilizing  (\ref{swj})  and (\ref{lmdwj}),
it can be easily checked that 
there are only two different eigenvalues concerning 
$   \mathbf  J (\mathbf  {w}_{j} )$,
which are   
\begin{equation}
0, 
\underbrace{
\frac
{ (n+1)  (m-1) }
{ 
	1+
	(-n)^{ m-2} n }  , 
\dots, 
\frac
{ (n+1)  (m-1) }
{ 
	1+
	(-n)^{ m-2} n }
}_{n-1}
. 
\end{equation}

By considering 
the odd and even cases for $m$ separately, 
it can be concluded that 
the  spectral radius of 
$   \mathbf  J (\mathbf  {w}_{j} )$
is given by 
\begin{align}\label{radiuscase}
\rho (\mathbf J (\mathbf  {w}_{j})) 
=
\begin{cases}
\frac
{ (n+1)  (m-1) }
{ n^{ m-1}-1 }    ,
 \quad   m = 2k+1, k \in \mathbb N
  \\
\frac
{ (n+1)  (m-1)}
{1+n^{ m-1}  }    , 
, \quad  m = 2k, k \in \mathbb N
\end{cases},
\end{align}
where 
for the odd $m$ cases, 
since $
	1+
(-n)^{ m-2} n
=	1-
n^{ m-1}<0$, 
the numerator 
should be  changed to be positive.  

Note that  
$n \ge 2$,  $m \ge 3$, 
it can be checked that
$n + m \ge 7$, 
it holds that 
$ \rho (\mathbf J (\mathbf  {w}_{j}))  < 1 $, 
and the 
vectors in the frame   of the  corresponding 
 regular  simplex  tensor  are  robust.
 
\end{proof}

\begin{remark}
It can be seen  from 
(\ref{swj})  and (\ref{lmdwj}) that when	$ \mathbf v = \mathbf w_{j}$,
we can 
explicitly  obtain the result of 
$ \mathcal{S}  \mathbf v^{m-2}  $
and  the  corresponding  eigenvalue  $\lambda$, 
and then we can  easily  determine the eigenvalues distribution
of $ \mathbf J (\mathbf  {w}_{j})$.
However,  for the other eigenpairs $ 
\mathbf v$, it will be difficult to 
determine the result of 
$ 
\mathbf v^{\mathrm T} \mathbf{w}_{i}$
in (\ref{swj})
and  that of $\lambda$, 
thus 
deriving  an explicit form 
concerning 
$ \mathbf J (\mathbf  {w}_{j})$.
This is the main difficulties for the strategy to check all eigenapirs   as  adopted  in the previous  work.
Therefore, the previous work only check the case of $n=2$, which could be the simplest situation. 
When $n$  becomes larger, it will be a complicated task.

	Note that 
(\ref{radiuscase})
has also been  derived 
in Theorem 4.5 
in \cite{RobustEigen}.
See the first formula in its  proof for details. 
Therefore, 
the sequential 
analysis 
after deriving  (\ref{radiuscase})
is the same as that for the proof of Theorem 4.5 
in \cite{RobustEigen},
and thus we omit these details in the above. 
However, 
the differences
lie in that the result 
presented in Theorem 4.5
is only a  upper-bounded result,
while here, we 
accurately 
determine  the spectral radius 
by analyzing its eigenvalues distributions. 
In other words, 
we further show that 
the spectral radius can reach at the bound in Theorem 4.5.

The above proof also confirms that 
 the  robust eigenpairs of   regular simplex tensors  do not exist
 for the case of $(m,n)=(2,3)$ and  $(m,n)=(3,3)$. This is also consistent with the experimental results 
 presented in Table  I
 in Ref 
 \cite{RobustEigen}.



\end{remark}
\section{Conclusion and  Future Work}\label{futurework}
In this paper, 
we mainly focus on a conjecture concerning the robust eigenpair of regular simplex tensor, 
and 
provide a proof for it. 
Different from the previous works, 
our proof  way first investigated  the relationship between 
robust eigenpairs and  locally maximized eigenpairs, 
which is another important concept in optimization field.
Then, such a established  relation
 will help us   to greatly simplify 
 the proof  progress of the focused conjecture  by only considering 
 a smaller subset of all  eigenpairs. 
 Benefiting from such a processing, 
 the conclusion in the conjecture is finally
 demonstrated.

 The other conjectures  claimed in the original reference \cite{RobustEigen}
are also worth investigating  in the next  stage. 
One can refer to Conjecture 4.8 and 5.2, Problem 5.1 for detailed  contents.

\bibliographystyle{unsrt}
\bibliography{simplexref}

\end{document}